\RequirePackage{fix-cm}
\documentclass[smallextended]{svjour3}       
\smartqed  
\usepackage{graphicx}
\usepackage{amsfonts}
\begin{document}

\title{Taylor's formula and related inequalities for a derivative with a new
parameter
}

\titlerunning{Inequalities for a derivative with a new parameter}        

\author{Deniz U\c{c}ar 
}


\institute{Deniz U\c{c}ar \at
              Department of Mathematics, Faculty of Sciences and Arts, U\c{s}ak
University, Turkey \\
              Tel.: +90-276-221-2121\\
              Fax: +90-276-221-2135\\
              \email{deniz.ucar@usak.edu.tr}           
          }

\date{Received: date / Accepted: date}

\maketitle

\begin{abstract}
In this paper, we derive Taylor's theorem for beta-fractional derivative. We
also investigate some new properties of Taylor's theorem and some useful
related theorems for this derivative. We extend some recent and classical
integral inequalities to this new simple interesting fractional calculus
including Steffensen and Hermit-Hadamard inequality.
\end{abstract}

\section{Introduction}
\label{intro}
Since L'Hospital in 1965 asked "What does $\frac{d^{n}f}{dx^{n}}$
mean if $n=\frac{1}{2}"~$to Lebniz, many researchers tried to define a
fractional derivative. Most of them defined integral form for the fractional
derivative. The most popular ones are:

(i) The Riemann-Liouville fractional derivative of a function $f$ is defined
as
\begin{eqnarray*}
D_{x}^{\alpha }\left( f\left( x\right) \right) =\frac{1}{\Gamma
\left( n-\alpha \right) }\left( \frac{d}{dx}\right)
^{n}\int\limits_{0}^{x}\left( x-t\right) ^{n-\alpha -1}f\left(
t\right) dt,~n-1<\alpha \leq n.
\end{eqnarray*}

(ii) Caputo's definition of fractional derivative is illustrated as follows
\begin{eqnarray*}
_{a}^{C}D_{t}^{\alpha }\left( f\left( t\right) \right)
=\frac{1}{\Gamma \left( n-\alpha \right) }\int\limits_{0}^{x}\left(
x-t\right) ^{n-\alpha -1}f^{\left( n\right) }\left( \tau \right)
d\tau ,~n-1<\alpha \leq n.
\end{eqnarray*}

(iii) The modified Liouville fractional derivative of a function $f$ is
defined as
\begin{eqnarray*}
D_{x}^{\alpha }\left( f\left( x\right) \right) =\frac{1}{\Gamma
\left( n-\alpha \right) }\left( \frac{d}{dx}\right)
^{n}\int\limits_{0}^{x}\left( x-t\right) ^{n-\alpha -1}\left(
f\left( t\right) -f\left( 0\right) \right) dt,~n-1<\alpha \leq n.
\end{eqnarray*}

(iv) \cite{d} The conformable fractional derivative of $f$ of order $\alpha $
is defined by
\begin{eqnarray*}
 f^{\left( \alpha \right) }\left( t\right) =\lim_{\varepsilon \to 0}\frac{f\left( t+\varepsilon t^{1-\alpha }\right) -f\left( t\right)
}{\varepsilon } \\
\end{eqnarray*}%
for $t>0,~\alpha \in \left( 0,1\right) $.

(v) \cite{f} The modified conformable fractional derivative is defined as
\begin{eqnarray*}
D^{\alpha }\left( f\right) \left( t\right) =\lim_{\varepsilon \to 0}\frac{f\left( te^{\varepsilon t^{-\alpha }}\right)
-f\left( t\right) }{\varepsilon }\\
\end{eqnarray*}%
for $t>0,~\alpha \in \left( 0,1\right) $.

For a review of this topic we direct the reader to the monograph \cite{a}.
However those fractional derivatives have some inconsistencies. In instance,
if $\alpha $ is not a natural number, most of the defined fractional
derivatives do not satisfy $D_{a}^{\alpha }\left( 1\right) =0.$ Some of the
fractional derivatives do not satisfy product rule for two functions. The
conformable and modified conformable fractional derivatives satisfy the
common properties of the standart rules but they have some limitations. We
can see the weakness of the defined fractional derivatives in \cite{f}.

A. Atangana et al in \cite{b} proposed a suitable derivative called the
Beta-derivative that allowed us to escape the lack of the fractional
derivatives. We use beta-fractional derivative introduced by Abdon Atangana
in \cite{b} to obtain our results.

\begin{definition}
Let $f$ be a function, such that $f:\left[ a,\infty \right) \rightarrow
\mathbb{R}.$ Then, the beta-derivative of a function $f$ is defined as%
\begin{eqnarray*}
_{0}^{A}D_{x}^{\beta }\left( f\left( x\right) \right) =\lim_{\varepsilon \to 0}\frac{f\left( x+\varepsilon \left( x+\frac{1}{\Gamma
\left( \beta \right) }\right) ^{1-\beta }\right) -f\left( x\right) }{%
\varepsilon },
\\
\end{eqnarray*}%
for all $x\geq a,$\bigskip\ $\beta \in \left( 0,1\right] .$ Then if the
limit exists, $f$ is said to be $\beta $-differentiable.
\end{definition}

There is a relation between $\beta $-derivative and usual derivative.%
\begin{eqnarray*}
_{0}^{A}D_{x}^{\beta }\left( f\left( x\right) \right) =\left( x+\frac{1}{%
\Gamma \left( \beta \right) }\right) ^{\beta -1}f^{^{\prime }}\left( x\right)
 \\
\end{eqnarray*}%
where $f^{^{\prime }}\left( x\right) =\lim_{\varepsilon \to 0}\frac{%
f\left( x+h\right) -f\left( x\right) }{h}.$

\begin{definition}
Let $f:\left[ a,b\right] \rightarrow \mathbb{R}$ be a continuous function on
the opened interval $\left( a,b\right) ,$ then the $\beta $-integral of f is
given as:%
\begin{eqnarray*}
_{0}^{A}I_{t}^{\beta }\left( f\left( t\right) \right) =\int_{0}^{t}\left( x+%
\frac{1}{\Gamma \left( \beta \right) }\right) ^{\beta-1}f\left( x\right) dx.
 \\
\end{eqnarray*}%
This integral was recently reffered to as the Atangana beta-integral.
\end{definition}
\begin{definition}
Let $\beta \in \left( 0,1\right] $ and $0\leq a<b.$ A function $f:\left[ a,b%
\right] \rightarrow \mathbb{R}$ is $\beta -fractional$ integrable on $\left[
a,b\right] $ if the integral
\begin{eqnarray*}
\int_{a}^{b}f\left( t\right) d_{\beta }t:=\int_{a}^{b}\left( t+\frac{1}{%
\Gamma \left( \beta \right) }\right) ^{\beta -1}f\left( t\right) dt
 \\
\end{eqnarray*}%
exists and finite.
\end{definition}

We assume the reader is familiar with the notation and basic results for
fractional calculus. For a review of this topic we direct the reader to the
monograph \cite{e,c,b,b1}.

\section{Main Results}
\label{sec:1}
In this section, we give the main theorem of the paper and obtain
some results close to the results in classical calculus. We first introduce
Taylor formula with a new parameter.
\begin{theorem}
\bigskip \label{Taylor Formula}Let $\beta \in \left( 0,1\right] $ and $n\in
\mathbb{N}$. If the function $f$ is $\left( n+1\right) $ order $\beta
-fractional$ differentiable on $\left[ 0,\infty \right) $ and $s,t\in \left[
0,\infty \right) $, then we have
\begin{eqnarray*}
f\left( t\right) &=&\sum\limits_{k=0}^{n}\frac{\beta
^{-k}}{k!}\left[
\left( t+\frac{1}{\Gamma \left( \beta \right) }\right) ^{\beta }-\left( s+%
\frac{1}{\Gamma \left( \beta \right) }\right) ^{\beta }\right]
^{k}D_{s}^{k\beta }f\left( s\right) \\
\end{eqnarray*}
\begin{eqnarray}
&&+\frac{\beta ^{-n}}{n!}\int_{s}^{t}\left[ \left( t+\frac{1}{\Gamma \left(
\beta \right) }\right) ^{\beta }-\left( \tau +\frac{1}{\Gamma \left( \beta
\right) }\right) ^{\beta }\right] ^{n}D_{\tau }^{\left( n+1\right) \beta
}f\left( \tau \right) d_{\beta }\tau .  \label{1}
\end{eqnarray}
\end{theorem}
\begin{proof}
Using integraton by parts, we have%
\begin{eqnarray*}
&&\frac{\beta ^{-n}}{n!}\int_{s}^{t}\left[ \left( t+\frac{1}{\Gamma \left(
\beta \right) }\right) ^{\beta }-\left( \tau +\frac{1}{\Gamma \left( \beta
\right) }\right) ^{\beta }\right] ^{n}D_{\tau }^{\left( n+1\right) \beta
}f\left( \tau \right) d_{\beta }\tau \\
&=&-\frac{\beta ^{-n}}{n!}\left[ \left( t+\frac{1}{\Gamma \left( \beta
\right) }\right) ^{\beta }-\left( s+\frac{1}{\Gamma \left( \beta \right) }%
\right) ^{\beta }\right] ^{n}D_{s}^{n\beta }f\left( s\right) \\
&&+\frac{\beta ^{1-n}}{\left( n-1\right) !}\int_{s}^{t}\left[ \left( t+\frac{%
1}{\Gamma \left( \beta \right) }\right) ^{\beta }-\left( \tau +\frac{1}{%
\Gamma \left( \beta \right) }\right) ^{\beta }\right] ^{n-1}D_{\tau
}^{n\beta }f\left( \tau \right) d_{\beta }\tau .
\end{eqnarray*}%
By the same way, integrating the second part of the right of equality, we
obtain%
\begin{eqnarray*}
&&\frac{\beta ^{1-n}}{\left( n-1\right) !}\int_{s}^{t}\left[ \left( t+\frac{1%
}{\Gamma \left( \beta \right) }\right) ^{\beta }-\left( \tau +\frac{1}{%
\Gamma \left( \beta \right) }\right) ^{\beta }\right] ^{n-1}D_{\tau
}^{n\beta }f\left( \tau \right) d_{\beta }\tau \\
&=&-\frac{\beta ^{1-n}}{\left( n-1\right) !}\left[ \left( t+\frac{1}{\Gamma
\left( \beta \right) }\right) ^{\beta }-\left( s+\frac{1}{\Gamma \left(
\beta \right) }\right) ^{\beta }\right] ^{n-1}D_{s}^{\left( n-1\right) \beta
}f\left( s\right)  \\
&&+\frac{\beta ^{2-n}}{\left( n-2\right) !}\int_{s}^{t}\left[ \left( t+\frac{%
1}{\Gamma \left( \beta \right) }\right) ^{\beta }-\left( \tau +\frac{1}{%
\Gamma \left( \beta \right) }\right) ^{\beta }\right] ^{n-2}D_{\tau
}^{\left( n-1\right) \beta }f\left( \tau \right) d_{\beta }\tau .  \
\end{eqnarray*}%
If we continue integrating by this way, we have%
\begin{eqnarray*}
&&\frac{\beta ^{-n}}{n!}\int_{s}^{t}\left[ \left( t+\frac{1}{\Gamma \left(
\beta \right) }\right) ^{\beta }-\left( \tau +\frac{1}{\Gamma \left( \beta
\right) }\right) ^{\beta }\right] ^{n}D_{\tau }^{\left( n+1\right) \beta
}f\left( \tau \right) d_{\beta }\tau  \\
&=&-\frac{\beta ^{-n}}{n!}\left[ \left( t+\frac{1}{\Gamma \left( \beta
\right) }\right) ^{\beta }-\left( s+\frac{1}{\Gamma \left( \beta \right) }%
\right) ^{\beta }\right] ^{n}D_{s}^{n\beta }f\left( s\right)  \\
&&-\frac{\beta ^{1-n}}{\left( n-1\right) !}\left[ \left( t+\frac{1}{\Gamma
\left( \beta \right) }\right) ^{\beta }-\left( s+\frac{1}{\Gamma \left(
\beta \right) }\right) ^{\beta }\right] ^{n-1}D_{s}^{\left( n-1\right) \beta
}f\left( s\right) \\
&&+...-\frac{\beta ^{-1}}{n!}\left[ \left( t+\frac{1}{\Gamma \left( \beta
\right) }\right) ^{\beta }-\left( s+\frac{1}{\Gamma \left( \beta \right) }%
\right) ^{\beta }\right] D_{s}^{\beta }f(s)+\int_{s}^{t}D_{\tau }^{\beta
}f\left( \tau \right) d_{\beta }\tau \\
&=&-\frac{\beta ^{-n}}{n!}\left[ \left( t+\frac{1}{\Gamma \left( \beta
\right) }\right) ^{\beta }-\left( s+\frac{1}{\Gamma \left( \beta \right) }%
\right) ^{\beta }\right] ^{n}D_{s}^{n\beta }f\left( s\right)  \\
&&-\frac{\beta ^{1-n}}{\left( n-1\right) !}\left[ \left( t+\frac{1}{\Gamma
\left( \beta \right) }\right) ^{\beta }-\left( s+\frac{1}{\Gamma \left(
\beta \right) }\right) ^{\beta }\right] ^{n-1}D_{s}^{\left( n-1\right) \beta
}f\left( s\right) \\
&&+...-\frac{\beta ^{-1}}{n!}\left[ \left( t+\frac{1}{\Gamma \left( \beta
\right) }\right) ^{\beta }-\left( s+\frac{1}{\Gamma \left( \beta \right) }%
\right) ^{\beta }\right] D_{s}^{\beta }f(s)+f(t)-f(s).  \
\end{eqnarray*}%
Thus we prove the equality (\ref{1}). We call the integral in the last
inequality $\beta -$Taylor Remainder of the function $f$.
\end{proof}
\begin{definition}
\bigskip \label{Taylor Remainder}Let $\beta \in \left( 0,1\right] $ and the
function $f$ is $\left( n+1\right) $ times $\beta -fractional$
differentiable on $\left[ 0,\infty \right) .$ Using (\ref{1}), we define the
remainder function by
\begin{eqnarray*}
R_{n.f}\left( s,t\right) :=f\left( s\right) -\sum\limits_{k=0}^{n}\frac{%
\beta ^{-k}}{k!}\left[ \left( t+\frac{1}{\Gamma \left( \beta \right) }%
\right) ^{\beta }-\left( s+\frac{1}{\Gamma \left( \beta \right) }\right)
^{\beta }\right] ^{k}D_{s}^{k\beta }f\left( s\right)  \\
\end{eqnarray*}%
and
\begin{equation}
R_{n.f}\left( s,t\right) =\frac{\beta ^{-n}}{n!}\int_{s}^{t}\left[ \left( t+%
\frac{1}{\Gamma \left( \beta \right) }\right) ^{\beta }-\left( \tau +\frac{1%
}{\Gamma \left( \beta \right) }\right) ^{\beta }\right] ^{n}D_{\tau
}^{\left( n+1\right) \beta }f\left( \tau \right) d_{\beta }\tau ,  \label{3}
\end{equation}%
for $n>-1.$
\end{definition}

\begin{theorem}
\label{odt}\bigskip Let $f$ and $g$ are continuous on the closed interval $%
\left[ a,b\right] $ and also $g\geq 0.$ Then there exists a point $c\in %
\left[ a,b\right] $ where
\begin{eqnarray*}
\int\limits_{a}^{b}f\left( t\right) g\left( t\right) d_{\beta
}t=f\left( c\right) \int\limits_{a}^{b}g\left( t\right) d_{\beta }t. \
\end{eqnarray*}
\end{theorem}
\begin{proof}
We define $m:={\min }f\left( t\right), M:={\max }f\left( t\right).$ So we have
\begin{equation}
m\leq f\left( t\right) \leq M.  \label{2}
\end{equation}%
Multiplying both sides of the inequality (\ref{2}) by the function $\left( t+%
\frac{1}{\Gamma \left( \beta \right) }\right) ^{\beta -1}g\left( t\right) $
and integrating over $\left( a,b\right) $ with respect to $t,$ we get
\begin{eqnarray*}
m\leq \frac{\int\limits_{a}^{b}f\left( t\right) g\left( t\right) d_{\beta }t%
}{\int\limits_{a}^{b}g\left( t\right) d_{\beta }t}\leq M  \\
\end{eqnarray*}%
where $\int\limits_{a}^{b}g\left( t\right) d_{\beta }t\neq 0.$ If $%
\int\limits_{a}^{b}g\left( t\right) d_{\beta }t=0,$ then we can
choose any point c. Since the function $f$ is continuous on $\left[
a,b\right] ,~f$ takes each value over $\left[ m,M\right] $ at least
once. So we have
\begin{eqnarray*}
f\left( c\right) =\frac{\int\limits_{a}^{b}f\left( t\right) g\left(
t\right) d_{\beta }t}{\int\limits_{a}^{b}g\left( t\right) d_{\beta
}t} \\
\end{eqnarray*}%
for some $c\in \left[ a,b\right] .$
\end{proof}

If we apply Theorem \ref{odt} to the $\beta -$Taylor Remainder (\ref{3}), we
have
\begin{eqnarray*}
R_{n.f}\left( s,t\right) &=&D_{t}^{\left( n+1\right) \beta }f\left( c\right)
\int_{s}^{t}\frac{\beta ^{-n}}{n!}\left[ \left( t+\frac{1}{\Gamma \left(
\beta \right) }\right) ^{\beta }-\left( \tau +\frac{1}{\Gamma \left( \beta
\right) }\right) ^{\beta }\right] ^{n}d_{\beta }\tau \\
&=&D_{t}^{\left( n+1\right) \beta }f\left( c\right) \frac{\beta ^{-n}}{n!}%
\left[ \left( t+\frac{1}{\Gamma \left( \beta \right) }\right) ^{\beta
}-\left( s+\frac{1}{\Gamma \left( \beta \right) }\right) ^{\beta }\right]
^{n+1}. \
\end{eqnarray*}%
We call this form of remainder $\beta -$Lagrange Remainder.

\begin{lemma}
\bigskip Let $\beta \in \left( 0,1\right] $ and the function $f$ is $\left(
n+1\right) $ times $\beta -fractional$ differentiable on $\left[ 0,\infty
\right) $. If $\beta -$Taylor's remainder is defined as in (\ref{3}), then
the following inequality holds.%
\begin{eqnarray*}
&&\int_{a}^{b}\frac{\beta ^{-n-1}}{\left( n+1\right) !}\left[ \left( t+\frac{%
1}{\Gamma \left( \beta \right) }\right) ^{\beta }-\left( \tau +\frac{1}{%
\Gamma \left( \beta \right) }\right) ^{\beta }\right] ^{n+1}D_{\tau
}^{\left( n+1\right) \beta }f\left( \tau \right) d_{\beta }\tau   \\
\end{eqnarray*}
\begin{eqnarray}
&=&\int_{a}^{t}R_{n.f}\left( a,\tau \right) d_{\beta }\tau
+\int_{t}^{b}R_{n.f}\left( b,\tau \right) d_{\beta }\tau .  \label{4}
\end{eqnarray}
\end{lemma}
\begin{proof}
We use mathematical induction to prove (\ref{4}). For $n=-1,$ we get
\begin{eqnarray*}
&&\int_{a}^{b}\left[ \left( t+\frac{1}{\Gamma \left( \beta \right) }\right)
^{\beta }-\left( \tau +\frac{1}{\Gamma \left( \beta \right) }\right) ^{\beta
}\right] ^{n+1}D_{\tau }^{\left( n+1\right) \beta }f\left( \tau \right)
d_{\beta }\tau  \\
&=&\int_{a}^{b}f\left( \tau \right) d_{\beta }\tau =\int_{a}^{t}f\left( \tau
\right) d_{\beta }\tau +\int_{t}^{b}f\left( \tau \right) d_{\beta }\tau . \
\end{eqnarray*}%
Assume that (\ref{4}) holds for $n=k-1,$
\begin{eqnarray*}
&&\int_{a}^{b}\frac{\beta ^{-k}}{k!}\left[ \left( t+\frac{1}{\Gamma \left(
\beta \right) }\right) ^{\beta }-\left( \tau +\frac{1}{\Gamma \left( \beta
\right) }\right) ^{\beta }\right] ^{k}D_{\tau }^{k\beta }f\left( \tau
\right) d_{\beta }\tau  \\
&=&\int_{a}^{t}R_{k-1.f}\left( a,\tau \right) d_{\beta }\tau
+\int_{t}^{b}R_{k-1.f}\left( b,\tau \right) d_{\beta }\tau . \
\end{eqnarray*}%
Let $n=k.$ Using the integration by parts, we obtain
\begin{eqnarray*}
&&\int_{a}^{b}\frac{\beta ^{-k-1}}{\left( k+1\right) !}\left[ \left( t+\frac{%
1}{\Gamma \left( \beta \right) }\right) ^{\beta }-\left( \tau +\frac{1}{%
\Gamma \left( \beta \right) }\right) ^{\beta }\right] ^{k+1}D_{\tau
}^{\left( k+1\right) \beta }f\left( \tau \right) d_{\beta }\tau  \label{5}   \\
&=&\frac{\beta ^{-k-1}}{\left( k+1\right) !}\left\{ \left[ \left( t+\frac{1}{%
\Gamma \left( \beta \right) }\right) ^{\beta }-\left( b+\frac{1}{\Gamma
\left( \beta \right) }\right) ^{\beta }\right] ^{k+1}D_{\tau }^{k\beta
}f\left( b\right) \right.   \\
&&-\left. \left[ \left( t+\frac{1}{\Gamma \left( \beta \right) }\right)
^{\beta }-\left( a+\frac{1}{\Gamma \left( \beta \right) }\right) ^{\beta }%
\right] ^{k+1}D_{\tau }^{k\beta }f\left( a\right) \right\}   \\
&&+\int_{a}^{b}\frac{\beta ^{-k}}{k!}\left[ \left( t+\frac{1}{\Gamma \left(
\beta \right) }\right) ^{\beta }-\left( \tau +\frac{1}{\Gamma \left( \beta
\right) }\right) ^{\beta }\right] ^{k}D_{\tau }^{k\beta }f\left( \tau
\right) d_{\beta }\tau .   \\
&=&\frac{\beta ^{-k-1}}{\left( k+1\right) !}\left\{ \left[ \left( t+\frac{1}{%
\Gamma \left( \beta \right) }\right) ^{\beta }-\left( b+\frac{1}{\Gamma
\left( \beta \right) }\right) ^{\beta }\right] ^{k+1}D_{\tau }^{k\beta
}f\left( b\right) \right.   \\
&&-\left. \left[ \left( t+\frac{1}{\Gamma \left( \beta \right) }\right)
^{\beta }-\left( a+\frac{1}{\Gamma \left( \beta \right) }\right) ^{\beta }%
\right] ^{k+1}D_{\tau }^{k\beta }f\left( a\right) \right\}   \\
&&+\int_{a}^{t}R_{k-1.f}\left( a,\tau \right) d_{\beta }\tau
+\int_{t}^{b}R_{k-1.f}\left( b,\tau \right) d_{\beta }\tau .
\end{eqnarray*}%
Since
\begin{eqnarray*}
&&\frac{\beta ^{-k}}{k!}D_{\tau }^{k\beta }f\left( b\right) \int_{b}^{t}%
\left[ \left( \tau +\frac{1}{\Gamma \left( \beta \right) }\right) ^{\beta
}-\left( b+\frac{1}{\Gamma \left( \beta \right) }\right) ^{\beta }\right]
^{k}d_{\beta }\tau   \\
&=&\frac{\beta ^{-k-1}}{\left( k+1\right) !}D_{\tau }^{k\beta }f\left(
b\right) \left[ \left( t+\frac{1}{\Gamma \left( \beta \right) }\right)
^{\beta }-\left( b+\frac{1}{\Gamma \left( \beta \right) }\right) ^{\beta }%
\right] ^{k+1}   \
\end{eqnarray*}%
and
\begin{eqnarray*}
&&\frac{\beta ^{-k}}{k!}D_{\tau }^{k\beta }f\left( a\right) \int_{a}^{t}%
\left[ \left( \tau +\frac{1}{\Gamma \left( \beta \right) }\right) ^{\beta
}-\left( a+\frac{1}{\Gamma \left( \beta \right) }\right) ^{\beta }\right]
^{k}d_{\beta }\tau   \\
&=&\frac{\beta ^{-k-1}}{\left( k+1\right) !}D_{\tau }^{k\beta }f\left(
a\right) \left[ \left( t+\frac{1}{\Gamma \left( \beta \right) }\right)
^{\beta }-\left( a+\frac{1}{\Gamma \left( \beta \right) }\right) ^{\beta }%
\right] ^{k+1}.   \
\end{eqnarray*}%
Thus, we obtain
\begin{eqnarray*}
&&\int_{a}^{b}\frac{\beta ^{-k-1}}{\left( k+1\right) !}\left[ \left( t+\frac{%
1}{\Gamma \left( \beta \right) }\right) ^{\beta }-\left( \tau +\frac{1}{%
\Gamma \left( \beta \right) }\right) ^{\beta }\right] ^{k+1}D_{\tau
}^{\left( k+1\right) \beta }f\left( \tau \right) d_{\beta }\tau  \\
&=&\int_{a}^{t}R_{k-1.f}\left( a,\tau \right) d_{\beta }\tau
+\int_{t}^{b}R_{k-1.f}\left( b,\tau \right) d_{\beta }\tau  \\
&&+\frac{\beta ^{-k}}{k!}D_{\tau }^{k\beta }f\left( b\right) \int_{b}^{t}%
\left[ \left( \tau +\frac{1}{\Gamma \left( \beta \right) }\right) ^{\beta
}-\left( b+\frac{1}{\Gamma \left( \beta \right) }\right) ^{\beta }\right]
^{k}d_{\beta }\tau  \\
&&-\frac{\beta ^{-k}}{k!}D_{\tau }^{k\beta }f\left( a\right) \int_{a}^{t}%
\left[ \left( \tau +\frac{1}{\Gamma \left( \beta \right) }\right) ^{\beta
}-\left( a+\frac{1}{\Gamma \left( \beta \right) }\right) ^{\beta }\right]
^{k}d_{\beta }\tau  \\
&=&\int_{a}^{t}\left\{ R_{k-1.f}\left( a,\tau \right) -\frac{\beta ^{-k}}{k!}%
D_{\tau }^{k\beta }f\left( a\right) \left[ \left( \tau +\frac{1}{\Gamma
\left( \beta \right) }\right) ^{\beta }-\left( a+\frac{1}{\Gamma \left(
\beta \right) }\right) ^{\beta }\right] ^{k}\right\} d_{\beta }\tau  \\
&&+\int_{t}^{b}\left\{ R_{k-1.f}\left( b,\tau \right) -\frac{\beta ^{-k}}{k!}%
D_{\tau }^{k\beta }f\left( b\right) \left[ \left( \tau +\frac{1}{\Gamma
\left( \beta \right) }\right) ^{\beta }-\left( b+\frac{1}{\Gamma \left(
\beta \right) }\right) ^{\beta }\right] ^{k}\right\} d_{\beta }\tau . \
\end{eqnarray*}%
This completes the proof.
\end{proof}

\begin{corollary}
\bigskip \bigskip \label{corollary2}Let $\beta \in \left( 0,1\right] .$ We
have%
\begin{eqnarray*}
&&\int_{a}^{b}\frac{\beta ^{-n-1}}{\left( n+1\right) !}\left[ \left( a+\frac{%
1}{\Gamma \left( \beta \right) }\right) ^{\beta }-\left( \tau +\frac{1}{%
\Gamma \left( \beta \right) }\right) ^{\beta }\right] ^{n+1}D_{\tau
}^{\left( n+1\right) \beta }f\left( \tau \right) d_{\beta }\tau  \\
&=&\int_{a}^{b}R_{n.f}\left( b,\tau \right) d_{\beta }\tau  \\
&&\int_{a}^{b}\frac{\beta ^{-n-1}}{\left( n+1\right) !}\left[ \left( b+\frac{%
1}{\Gamma \left( \beta \right) }\right) ^{\beta }-\left( \tau +\frac{1}{%
\Gamma \left( \beta \right) }\right) ^{\beta }\right] ^{n+1}D_{\tau
}^{\left( n+1\right) \beta }f\left( \tau \right) d_{\beta }\tau  \\
&=&\int_{a}^{b}R_{n.f}\left( a,\tau \right) d_{\beta }\tau .  \
\end{eqnarray*}
\end{corollary}

\subsection{Steffensen Inequality}
\label{sec:2}
We prove a new $\beta -fractional$ version of Steffensen inequality
and of Hayashi's inequality. We need the folllowing lemma to prove our
results.

\begin{lemma}
\bigskip \bigskip Let $\beta \in \left( 0,1\right] $ and $a,~b\in \mathbb{R}$
with $0\leq a<b.$ We assume $M>0$ and $f:\left[ a,b\right] \rightarrow \left[
0,M\right] $ be an $\beta -fractional$ integrable function on $\left[ a,b%
\right] .$ Then the inequalities%
\begin{equation}
\int_{b-l}^{b}Md_{\beta }t\leq \int_{a}^{b}f\left( t\right) d_{\beta }t\leq
\int_{a}^{a+l}Md_{\beta }t  \label{6}
\end{equation}%
hold where%
\begin{equation}
l:=\frac{\beta \left( b-a\right) }{M\left[ \left( b+\frac{1}{\Gamma \left(
\beta \right) }\right) ^{\beta }-\left( a+\frac{1}{\Gamma \left( \beta
\right) }\right) ^{\beta }\right] }\int_{a}^{b}f\left( t\right) d_{\beta }t,%
 t\in \left[ 0,b-a\right] .  \label{7}
\end{equation}
\end{lemma}

\begin{proof}
Since $f\left( t\right) \in \left[ 0,M\right] $ for all $t\in \left[ a,b%
\right] ,$ using (\ref{7}) we have%
\begin{eqnarray*}
0 &\leq &l=\frac{\beta \left( b-a\right) }{M\left[ \left( b+\frac{1}{\Gamma
\left( \beta \right) }\right) ^{\beta }-\left( a+\frac{1}{\Gamma \left(
\beta \right) }\right) ^{\beta }\right] }\int_{a}^{b}f\left( t\right)
d_{\beta }t  \\
&\leq &\frac{\beta \left( b-a\right) }{\left( b+\frac{1}{\Gamma \left( \beta
\right) }\right) ^{\beta }-\left( a+\frac{1}{\Gamma \left( \beta \right) }%
\right) ^{\beta }}\int_{a}^{b}d_{\beta }t=b-a.  \
\end{eqnarray*}%
We can easily see that $\left( t+\frac{1}{\Gamma \left( \beta \right) }%
\right) ^{\beta -1}$ is a decreasing function on $\left[ a,b\right] $\ or $%
\left( a,b\right] $ for $a=0.$ Thus using the fact that $d_{\beta }t=\left( t+%
\frac{1}{\Gamma \left( \beta \right) }\right) ^{\beta -1},$ we obtain the
following inequalities
\begin{eqnarray*}
\frac{1}{l}\int_{b-l}^{b}d_{\beta }t\leq \frac{1}{b-a}\int_{a}^{b}d_{\beta
}t\leq \frac{1}{l}\int_{a}^{a+l}d_{\beta }t.  \\
\end{eqnarray*}%
So that using (\ref{7}), we have
\begin{eqnarray*}
\int_{b-l}^{b}Md_{\beta }t\leq \frac{l}{b-a}\int_{a}^{b}Md_{\beta }t\leq
\int_{a}^{a+l}Md_{\beta }t  \\
\end{eqnarray*}%
which ends the proof.
\end{proof}

We call the following theorem as Steffensen's inequality. If we
choose $M=1$ and $A>0$ for generality it is called Hayashi's inequality.
\begin{theorem}
\label{Fractional Hayashi-Steffensen Inequality}\bigskip Let $\beta \in
\left( 0,1\right] $ and $a,~b\in \mathbb{R}$ with $0\leq a<b$ and $M>0.$
Assume that the functions $f:\left[ a,b\right] \rightarrow \left[ 0,M\right]
$ and $g:\left[ a,b\right] \rightarrow \left[ 0,M\right] $\ be $\beta
-fractional$ integrable functions on $\left[ a,b\right] .$ If $f$ is
nonnegative and nonincreasing, then
\begin{equation}
M\int_{b-l}^{b}f\left( t\right) d_{\beta }t\leq \int_{a}^{b}f\left( t\right)
g\left( t\right) d_{\beta }t\leq M\int_{a}^{a+l}f\left( t\right) d_{\beta }t
\label{8}
\end{equation}%
where $l$ is defined in (\ref{7}).
\end{theorem}
\begin{proof}
Assume that $f$ is nonnegative and nonincreasing, we first prove the left
side of the inequality. By the definition of $l$ in (\ref{7}) and the
conditions on function $g,$ we have inequality (\ref{6}). Using the left
hand side of the inequality (\ref{8}), we obtain
\begin{eqnarray*}
&&\int_{a}^{b}f\left( t\right) g\left( t\right) d_{\beta
}t-M\int_{b-l}^{b}f\left( t\right) d_{\beta }t  \\
&=&\int_{a}^{b-l}f\left( t\right) g\left( t\right) d_{\beta
}t+\int_{b-l}^{b}f\left( t\right) g\left( t\right) d_{\beta
}t-M\int_{b-l}^{b}f\left( t\right) d_{\beta }t  \\
&=&\int_{a}^{b-l}f\left( t\right) g\left( t\right) d_{\beta
}t-\int_{b-l}^{b}\left( M-g\left( t\right) \right) f\left( t\right) d_{\beta
}t.  \
\end{eqnarray*}%
Since $f$ is nonincreasing and for $t\in \left[ b-l,b\right] ,$ we have
\begin{equation}
f\left( t\right) \leq f\left( b-l\right)  \label{9}
\end{equation}%
Using the inequality (\ref{9}), we get
\begin{eqnarray*}
&&\int_{a}^{b-l}f\left( t\right) g\left( t\right) d_{\beta
}t-\int_{b-l}^{b}\left( M-g\left( t\right) \right) f\left( t\right) d_{\beta
}t  \\
&\geq &\int_{a}^{b-l}f\left( t\right) g\left( t\right) d_{\beta }t-f\left(
b-l\right) \int_{b-l}^{b}\left( M-g\left( t\right) \right) d_{\beta }t  \\
&\geq &\int_{a}^{b-l}f\left( t\right) g\left( t\right) d_{\beta }t-f\left(
b-l\right) \int_{b-l}^{b}g\left( t\right) d_{\beta }t  \\
&\geq &\int_{a}^{b-l}\left( f\left( t\right) -f\left( b-l\right) \right)
g\left( t\right) d_{\beta }t\geq 0.  \
\end{eqnarray*}%
Thus we prove the left hand side of the inequality (\ref{8}). The proof of
the right hand side of the inequality is similar and one can easily prove
theinequality by using (\ref{6}).
\end{proof}

\begin{theorem}
If $f$ is nonpositive and nondecreasing function and $g:\left[ a,b\right]
\rightarrow \left[ 0,M\right] ,$\ the inequalities in Theorem \bigskip \ref%
{Fractional Hayashi-Steffensen Inequality} are reversed.
\end{theorem}

\begin{proof}
Assume $f$ is nonpositive and nondecreasing. In this case, we prove the
right hand side of the inequality (\ref{8}). Using the inequality (\ref{8})
and the inequality (\ref{6}), we have
\begin{eqnarray*}
&&\int_{a}^{b}f\left( t\right) g\left( t\right) d_{\beta
}t-M\int_{a}^{a+l}f\left( t\right) d_{\beta }t  \\
&=&\int_{a}^{a+l}f\left( t\right) g\left( t\right) d_{\beta
}t+\int_{a+l}^{b}f\left( t\right) g\left( t\right) d_{\beta
}t-M\int_{a}^{a+l}f\left( t\right) d_{\beta }t  \\
&=&\int_{a+l}^{b}f\left( t\right) g\left( t\right) d_{\beta
}t+\int_{a}^{a+l}\left( g\left( t\right) -M\right) f\left( t\right) d_{\beta
}t  \\
&\geq &\int_{a+l}^{b}f\left( t\right) g\left( t\right) d_{\beta }t+f\left(
a+l\right) \int_{a}^{a+l}\left( g\left( t\right) -M\right) d_{\beta }t  \\
&\geq &\int_{a+l}^{b}f\left( t\right) g\left( t\right) d_{\beta }t-f\left(
a+l\right) \int_{a}^{a+l}g\left( t\right) d_{\beta }t  \\
&=&\int_{a+l}^{b}\left( f\left( t\right) -f\left( a+l\right) \right) g\left(
t\right) d_{\beta }t\geq 0.  \
\end{eqnarray*}%
Thus the right hand side of the inequality (\ref{8}) holds. The proof of the
left hand side of the reversed inequality is similar as in the proof of the
Theorem \ref{Fractional Hayashi-Steffensen Inequality}.
\end{proof}

In the following we obtain some results by using \bigskip $\beta -fractional$
Steffensen inequality.

\begin{theorem}
\label{theorem5}\bigskip Let $\beta \in \left( 0,1\right] $ and the function
$f:\left[ a,b\right] \rightarrow \mathbb{R}$ be $\left( n+1\right) $ times $%
\beta -fractional$ differentiable. We assume that $D^{\left( n+1\right)
\beta }f$ is increasing and $D^{n\beta }f$ is decreasing on $\left[ a,b%
\right] .$ Then inequalities
\begin{eqnarray*}
D^{n\beta }f\left( a+l\right) -D^{n\beta }f\left( a\right) &\leq &\left(
n+1\right) !\beta ^{n+1}\left[ \left( b+\frac{1}{\Gamma \left( \beta \right)
}\right) ^{\beta }-\left( a+\frac{1}{\Gamma \left( \beta \right) }\right)
^{\beta }\right] ^{-n-1}\int_{a}^{t}R_{n.f}\left( a,\tau \right) d_{\beta
}\tau  \\
&\leq &D^{n\beta }f\left( b\right) -D^{n\beta }f\left( b-l\right)  \
\end{eqnarray*}%
hold where
\begin{eqnarray*}
l:=\frac{b-a}{n+2}.  \\
\end{eqnarray*}
\end{theorem}

\begin{proof}
We define the function $F:=-D^{\left( n+1\right) \beta }f.$ Since $D^{\left(
n+1\right) \beta }f$ is increasing and $D^{n\beta }f$ is decreasing on $%
\left[ a,b\right] ,$ we have $D^{\left( n+1\right) \beta }f$ $\leq 0.$ So
that $F\geq 0$ and decreasing on $\left[ a,b\right] .$ For the assumptions
of Steffensen's inequality we define
\begin{eqnarray*}
g\left( t\right) :=\frac{\left[ \left( b+\frac{1}{\Gamma \left( \beta
\right) }\right) ^{\beta }-\left( t+\frac{1}{\Gamma \left( \beta \right) }%
\right) ^{\beta }\right] ^{n+1}}{\left[ \left( b+\frac{1}{\Gamma \left(
\beta \right) }\right) ^{\beta }-\left( a+\frac{1}{\Gamma \left( \beta
\right) }\right) ^{\beta }\right] ^{n+1}}  \\
\end{eqnarray*}%
for $t\in \left[ a,b\right] ,~n\geq -1.$ We apply $M=1$ in Theorem \bigskip %
\ref{Fractional Hayashi-Steffensen Inequality}, the functions $F$ and $g$
satisfy the assumptions of Steffensen's inequality. We can write
\begin{equation}
-\int_{b-l}^{b}D_{t}^{\left( n+1\right) \beta }f\left( t\right) d_{\beta
}t\leq -\int_{a}^{b}D_{t}^{\left( n+1\right) \beta }f\left( t\right) g\left(
t\right) d_{\beta }t\leq -\int_{a}^{a+l}D^{\left( n+1\right) \beta }f\left(
t\right) d_{\beta }t.  \label{10}
\end{equation}%
If we simplify (\ref{10}) using \bigskip Corollary \ref{corollary2}, we
obtain
\begin{eqnarray*}
D^{n\beta }f\left( a+l\right) -D^{n\beta }f\left( a\right) &\leq
&\left( n+1\right) !\beta ^{n+1}\left[ \left( b+\frac{1}{\Gamma
\left( \beta \right) }\right) ^{\beta }-\left( a+\frac{1}{\Gamma
\left( \beta \right) }\right) ^{\beta }\right]
^{-n-1}\int_{a}^{t}R_{n.f}\left( a,\tau \right) d_{\beta
}\tau  \\
&\leq &D^{n\beta }f\left( b\right) -D^{n\beta }f\left( b-l\right).  \
\end{eqnarray*}%
This completes the proof.
\end{proof}

\begin{remark}
\bigskip Let $D^{\left( n+1\right) \beta }f$ is decreasing and $D^{n\beta }f$
is increasing on $\left[ a,b\right] .$ If we define the function $%
F:=D^{\left( n+1\right) \beta }f,$ one can easily show that the above
inequalities in Theorem \ref{theorem5} are reversed.
\end{remark}

If we choose $n=0$ in Theorem \ref{theorem5}, we obtain the
well-known inequality called Hermite-Hadamard.

\begin{definition}
\bigskip \bigskip Let $\beta \in \left( 0,1\right] $ and the function $f:%
\left[ a,b\right] \rightarrow \mathbb{R}$ be $\beta -fractional$
differentiable. If $D^{\beta }f$ is increasing and $f$ is decreasing on $%
\left[ a,b\right] ,$ then we have the following Hermite-Hadamard inequality
\begin{equation}
f\left( \frac{a+b}{2}\right) \leq \frac{\beta }{\left( b+\frac{1}{\Gamma
\left( \beta \right) }\right) ^{\beta }-\left( a+\frac{1}{\Gamma \left(
\beta \right) }\right) ^{\beta }}\int_{a}^{b}f\left( t\right) d_{\beta
}t\leq f\left( b\right) +f\left( a\right) -f\left( \frac{a+b}{2}\right) .
\label{11}
\end{equation}
\end{definition}

\begin{remark}
\bigskip If $D^{\beta }f$ is decreasing and $f$ is increasing on $\left[ a,b%
\right] ,$ then the inequalities (\ref{11}) are reversed.
\end{remark}

\begin{corollary}
In this paper, we have studied Taylor formula with a new parameter and have
obtained new results for beta-fractional derivative. The main theorem
improves previously results and this presents a new approach to $\beta
-fractional$ version of Steffensen inequality and well known
Hermite-Hadamard inequality.
\end{corollary}


%
%



\end{document}